\documentclass{article}
\usepackage{amsfonts}
\usepackage{amsmath}

\newcommand{\ben}{\begin{equation}} 
\newcommand{\een}{\end{equation}}
\newcommand{\bem}{\begin{eqnarray}} 
\newcommand{\eem}{\end{eqnarray}}

\newcommand{\zd}{\ensuremath{\mathbb{Z}^{d}}} 
\newcommand{\rd}{\ensuremath{\mathbb{R}^{d}}} 
\newcommand{\cd}{\ensuremath{\mathbb{C}^{d}}} 

\newcommand{\mat}{Mat\'{e}rn}

\newtheorem{theorem}{Theorem}[section]

\newtheorem{corollary}[theorem]{Corollary}

\newtheorem{lemma}[theorem]{Lemma}

\newenvironment{proof}[1][Proof]{\noindent\textbf{#1.} }{\ \rule{0.5em}{0.5em}}
\begin{document}

\title{Uniformly bounded Lebesgue constants 
for scaled cardinal interpolation with \mat\ kernels\thanks{This work was supported by Kuwait
University, Research Grant No.\ SM01/18.}}
\author{Aurelian Bejancu\thanks{Department of Mathematics, Kuwait University, 
PO Box 5969, Safat 13060, Kuwait. \quad E-mail address: aurelianbejancu@gmail.com}
\thanks{Declarations of interest: none.}}
\maketitle

\begin{abstract}
For $h>0$ and positive integers $m$, $d$, such that $m>d/2$, we study non-stationary  
interpolation at the points of the scaled grid $h\zd$ via the \mat\ kernel $\Phi_{m,d}$---the 
fundamental solution of $(1-\Delta)^m$ in \rd. We prove that the Lebesgue constants of the 
corresponding interpolation operators are uniformly bounded as $h\to0$ and deduce the  
convergence rate $O(h^{2m})$ for the scaled interpolation scheme. We also provide convergence 
results for approximation with \mat\ and related compactly supported polyharmonic kernels.
\smallskip

\smallskip

\end{abstract}

\section{Introduction}

Cardinal interpolation at the points of the lattice \zd\ provides an ideal model for studying multivariable 
kernel interpolation, which extends Schoenberg's theory of univariate cardinal spline interpolation 
\cite{S73}. For a kernel $\phi:\rd\to\mathbb{R}$ decaying sufficiently fast, let
\[
\mathcal{S} (\phi) := \{ \sum_{j\in\zd} c_j \,\phi (\cdot - j) : c\in\ell_\infty \} .
\]
The problem of cardinal interpolation with $\phi$ is to find, for any bounded sequence of data values 
$\{y_j\}_{j\in\zd}$, a function $s\in\mathcal{S} (\phi)$, such that $s(j)=y_j$, for all $j\in\zd$. If such a 
function exists and it is unique, cardinal interpolation with $\phi$ is deemed to be `correct', or `solvable'.

One way to analyze the approximation properties of such an interpolation method is to consider the 
associated \emph{stationary scheme} for interpolation on the scaled lattice $h\zd$, for a scaling 
parameter $h>0$. For this, consider the space of dilations 
\[
\mathcal{S}^h (\phi) = \{ s (\cdot / h) : s\in\mathcal{S} (\phi) \} ,
\]
and, for a bounded data function $f:\rd\to\mathbb{R}$, let $s_{f,h}\in\mathcal{S}^h (\phi)$ be the unique 
bounded interpolant to the values of $f$ on $h\zd$. For $k\geq0$, this interpolation method is said to 
achieve the $L_\infty$\emph{-approximation order} $k$ if, for any sufficiently smooth 
$f\in L_\infty:=L_\infty(\rd)$, we have $\|f-s_{f,h}\|_\infty=O(h^k)$, as $h\to0$, 
with the $L_\infty$-norm $\|\!\cdot\!\|_\infty$.
The last condition implies, in particular, that the $L_\infty$-distance from any such $f$ to the space 
$\mathcal{S}^h (\phi)$ also decays at the rate $O(h^k)$, hence, by a well-known result (de Boor and 
Ron \cite[Theorem 3.1]{br92}), $\phi$ must satisfy the Strang-Fix conditions of order $k$, i.e.\ its 
Fourier transform $\widehat{\phi}$ must have a zero of order $k$ at every $j\in2\pi\zd\setminus0$ 
(for $k=1$, cf.\ Buhmann \cite[Theorem 23]{mdb90}). 

In case $\phi$ is a \emph{box-spline} kernel, the approximation order of stationary cardinal 
interpolation can be elegantly expressed in terms of the direction matrix defining $\phi$ (see the 
monograph by de Boor et al.\ \cite{bhr93}). But, in general, decaying kernels (e.g.\ Gaussian, 
generalized multiquadric, or \mat\ kernels) may not satisfy the Strang-Fix conditions. In such cases, 
one may consider a \emph{non-stationary} scheme, where dilations are selected from a space 
$\mathcal{S} (\phi_h)$ based on a $h$-dependent kernel $\phi_h$, usually itself a dilation of 
$\phi$. A natural choice is $\phi_h := \phi(h\cdot)$, in which case the collection 
$\{ \mathcal{S}^h (\phi_h) \}_h$---the \emph{flat ladder} generated by $\phi$---is defined as:
\ben
\mathcal{S}^h (\phi_h) := \{ \sum_{j\in\zd} c_j \,\phi_h (h^{-1}\cdot - j) : c\in\ell_\infty \} 
= \{ \sum_{j\in\zd} c_j \,\phi (\cdot - hj) : c\in\ell_\infty \} .
\label{eq:FL}
\een

Recently, the approximation order of flat ladder interpolation on $h\zd$ has been studied, for the 
\emph{Gaussian} kernel, by Hangelbroek et al.\ \cite{hmnw12}, and, for the \emph{generalized multiquadric} 
kernel, by Hamm and Ledford \cite{hl18}. The analysis employed in these two works makes use of an 
intermediate band-limited interpolant and ultimately relies on the fact that the Fourier transform 
$\widehat{\phi}$ of the Gaussian or the generalized multiquadric kernel decays exponentially at infinity. 

Here, we propose a different method, based on bounding uniformly the associated Lebesgue 
constants, in order to obtain the rate of approximation of the flat ladder interpolation scheme with the 
\emph{\mat} kernel, whose Fourier transform decays only algebraically at infinity. For a positive integer 
$m>d/2$, the \mat\ kernel $\phi=\Phi_{m,d}$ is defined, up to a constant factor, as the fundamental 
solution of the elliptic operator $(1-\Delta)^m$ in \rd, where $\Delta$ is the Laplace operator. This kernel 
is commonly used as a covariance function in statistical modeling 
(e.g.\ Gneiting et al.\ \cite{gks10}, Chen et al.\ \cite{cwa14}). 

The basic properties of cardinal interpolation on \zd\ with the \mat\ kernel $\Phi_{m,d}$ are provided in our 
recent work \cite[Example 5.2]{ab20}, which also covers the case of non-integer $m$, as well as the related 
model of `semi-cardinal' interpolation on half-space lattices. In particular, the corresponding Lagrange 
function $\chi\in\mathcal{S} (\Phi_{m,d})$, satisfying $\chi(j) = \delta_{j0}$ for all $j\in\zd$, is shown to decay 
exponentially at infinity.

In the present paper, we establish a much stronger version of this result (Theorem \ref{th:main}), by proving 
that, for $\phi=\Phi_{m,d}$, the exponential decay of the Lagrange function for cardinal interpolation on \zd\ 
from the space $\mathcal{S} (\phi_h)$ holds with constants that are independent of the scale parameter $h$. 
As a direct consequence (Corollary \ref{cor:Lebesgue}), we derive a scale independent bound on the Lebesgue 
constant for interpolation on $h\zd$ from the space $\mathcal{S}^h (\phi_h)$ (no similar result in the non-stationary 
setting seems to have been obtained previously in the literature). This, in turn, allows us to deduce the 
convergence rate $O(h^{2m})$ for the \mat\ flat ladder interpolation scheme on $h\zd$ from the corresponding 
rate for \emph{approximation} in $\mathcal{S}^h (\phi_h)$ implied by the work of de Boor and Ron \cite{br92}. 

Section \ref{sec:CI} contains preliminary material on cardinal interpolation with \mat\ kernels, including 
the Fourier representation of the corresponding Lagrange functions. Section \ref{sec:main} proves the main 
result of the paper, Theorem \ref{th:main}, while the convergence results are obtained in section \ref{sec:conv}. 
In section \ref{sec:comp-supp}, the rate of approximation $O(h^{2m})$ by \emph{finite} linear combinations of 
shifted \mat\ kernels is also transferred to the class of compactly supported `perturbation' kernels defined by 
Ward and Unser \cite{WU14}, and two families of polyharmonic radial kernels constructed by Johnson 
\cite{mjj12} are shown to belong to this class.

\section{Cardinal interpolation with \mat\ kernels}
\label{sec:CI}
\addtocounter{equation}{-1}

\textbf{\mat\ kernels and their basic properties.} 
For an integer $m$, such that $m>\frac{d}{2}$, the \mat\ kernel $\Phi:=\Phi_{m,d}$ is expressed as 
\ben
\Phi (x) = \|x\|^{m-\frac{d}{2}} K_{m-\frac{d}{2}} (\|x\|), \quad x\in\rd, 
\label{eq:matern}
\een
where $K_{\nu}$ denotes the modified Bessel function of the third kind of order $\nu$. It is known 
that $\Phi$ is continuous on \rd\ and $\Phi (x) = O ( \|x\|^{m-\frac{d+1}{2}} e^{ -\|x\| } )$, as $\|x\|\to\infty$, 
hence, there exists $\alpha\in(0,1]$ and $C_0:=C_0(m,d)>0$, such that 
\ben
|\Phi ( x )| \leq C_0\, e^{-\alpha \|x\|}, \quad  x\in \rd.
\label{eq:kernel-decay}
\een
It follows from \cite[Theorem 6.13]{wend05} that the Fourier transform of $\Phi$ is given by 
\ben
\widehat{\Phi} (t) = \rho_{m,d} (1+\|t\|^2)^{-m}, \quad t\in\rd,
\label{eq:kernel-FT}
\een
for a constant $\rho_{m,d} > 0$.
\medskip

\noindent
\textbf{Non-stationary interpolation scheme.} 
For each parameter $h>0$, we let 
\ben
\Phi_h (x) := \Phi (hx),\quad x\in\rd,
\label{eq:flat-scale}
\een
and define the associated scaled shift-invariant space $\mathcal{S}^h (\Phi_h)$ via (\ref{eq:FL}).
Since the exponential decay property (\ref{eq:kernel-decay}) ensures that
\begin{equation}
\sum_{j\in \zd} \sup_{y\in [0,1]^d}  | \Phi_h (y-j) | < \infty,
\label{eq:minimal}
\end{equation}
each element of $\mathcal{S}^h (\Phi_h)$ is a continuous function on \rd, being defined by a 
series which converges uniformly on compact sets. Also, such an element is a bounded function 
on \rd. The collection $\{ \mathcal{S}^h (\Phi_h) \}_h$ is the non-stationary flat ladder 
generated by $\Phi$. 

The main problem addressed in this paper is to interpolate a data function at the points of the scaled 
grid $h\zd$ from the space $\mathcal{S}^h (\Phi_h)$, i.e.\ using series representations of 
$h\zd$-translates of $\Phi$ with bounded coefficients. Due to translation invariance, this problem 
amounts to the construction of a \emph{Lagrange function} $\chi_h\in\mathcal{S}^h (\Phi_h)$, which 
satisfies $\chi_h (hj) = \delta_{j0}$, $j\in\zd$. Note that $\chi_h$ indicates a generic dependence on $h$, 
while the specific notation of (\ref{eq:flat-scale}) applies only to $\Phi_h$.

In the sequel, it is convenient to employ the change of variables $y=h^{-1}x$, by which the above 
problem is equivalently formulated as cardinal interpolation at the points of the lattice \zd\ 
from the corresponding shift-invariant space 
\[
\mathcal{S} (\Phi_h) = \{ \sum_{j\in\zd} c_j \,\Phi_h (\cdot - j) : c\in\ell_\infty \} .
\]
The existence and uniqueness of a solution to the latter problem 
(e.g.\ Chui et al.\ \cite[Lemma~1.1]{cjw87}) depend on the non-vanishing of the \emph{symbol} 
function defined by the absolutely convergent Fourier series:
\ben
\sigma_m (t,h) := \sum_{k\in \zd} \Phi_h (k) e^{ikt}, \quad t\in\rd.
\label{eq:CI-symbol}
\een
Note that, by (\ref{eq:kernel-FT}) and usual transform laws, we have 
\ben
\widehat{\Phi_h} (t) = h^{-d} \widehat{\Phi} (h^{-1} t) 
=  \frac {\rho_{m,d} h^{2m-d}} {(h^2+\|t\|^2)^{m}}, \quad t\in\rd. 
\label{eq:kernel-FTh}
\een
Hence, since $2m>d$, an application of the Poisson Summation Formula provides 
\bem
\sigma_m (t,h) & = & \sum_{k\in\zd} \widehat{\Phi_h} (t+2\pi k) \nonumber \\
& = & \rho_{m,d} h^{2m-d} \sum_{k\in\zd} (h^2+\|t+2\pi k\|^2)^{-m} \ > \ 0, 
\label{eq:PSF}
\eem
for all $t\in\rd$ and $h>0$, the series being uniformly convergent on compact sets.

It follows that cardinal interpolation on \zd\ with $\Phi_h$ is `correct' and the corresponding Lagrange 
function $\widetilde{\chi_h}$ satisfying $\widetilde{\chi_h} (j) = \delta_{j0}$, $j\in\zd$, is given by
\ben
\widetilde{\chi_h} (y) = \sum_{k\in\zd} a^h_k \,\Phi_h (y-k),\quad y\in\rd,
\label{eq:Lag-rescaled}
\een
where, by Wiener's lemma,  $\{a^h_k\}_{k\in\zd}\in\ell_1$ is the sequence of Fourier coefficients of 
$1/\sigma_m(\cdot,h)$. Hence, reverting to $x=hy$, the above Lagrange function $\chi_h$ for interpolation 
with $\Phi$ on $h\zd$ can be identified as 
\[
\chi_h (x) = \widetilde{\chi_h} (h^{-1}x) = \sum_{k\in\zd} a^h_k \,\Phi (x-hk),\quad x\in\rd.
\]

The following lemma shows that the Fourier transform of $\widetilde{\chi_h}$ 
is the function $\omega_m (\cdot,h) (h^2+\|\!\cdot\!\|^2)^{-m}$, where, by (\ref{eq:PSF}),   
\ben
\omega_m (t,h) := \frac{1}{\sum_{k\in\zd} (h^2+\|t+2\pi k\|^2)^{-m}} 
= \frac{\rho_{m,d} h^{2m-d}}{\sigma_m (t,h)} , 
\quad t\in\rd.
\label{eq:inv-symbol-r}
\een

\begin{lemma}		\label{le:Lag-FT}
For each $h>0$, $\omega_m (\cdot,h)$ is a continuous positive-valued function, $2\pi$-periodic in 
each of its $d$ variables, with the inverse Fourier representation:
\ben
\widetilde{\chi_h} ( y ) = \frac{1}{(2\pi)^d} \int_{\rd} e^{iyt} 
\frac {\omega_m (t,h)} {(h^2+\|t\|^2)^m} dt, \quad  y\in\rd.
\label{eq:Lag-FT}
\een
\end{lemma}

\begin{proof}
The continuity of $\omega_m (\cdot,h)$ is a consequence of the uniform convergence 
of the series (\ref{eq:PSF}) on compact sets. By (\ref{eq:kernel-FTh}), we have 
$\widehat{\Phi_h} \in L_1 (\rd)$. Hence, using the Fourier inversion formula and the fact that 
$\{a^h_k\}_{k\in\zd}$ is absolutely summable, (\ref{eq:Lag-rescaled}) implies 
\begin{eqnarray*}
\widetilde{\chi_h} (y) 
& = & \frac{1}{(2\pi)^d} \sum_{k\in\zd} a^h_k \int_{\rd} e^{it(y-k)} \, \widehat{\Phi_h} (t) \, dt  \\ 
& = & \frac{1}{(2\pi)^d} \int_{\rd} e^{iyt} \, \widehat{\Phi_h} (t) \sum_{k\in\zd} a^h_k \, e^{-itk} \, dt, 
\quad y\in\rd.
\end{eqnarray*}
Since $\sum_{k\in\zd} a^h_k \, e^{-itk} = [\sigma_m(t,h)]^{-1}$, the required representation follows 
via (\ref{eq:kernel-FTh}) and (\ref{eq:inv-symbol-r}).
\end{proof}
\medskip

\noindent
\textbf{Remark.} 
The inverse symbol $\omega_m$ is well-defined by the middle fraction of (\ref{eq:inv-symbol-r}) 
even for $h=0$. In this case, $\omega_m$ acquires a zero at $t=0$ of the same order as the 
denominator of the integrand in (\ref{eq:Lag-FT}). This integral representation, for $h=0$, was 
used by Madych and Nelson \cite{mn90} as definition of the Lagrange 
function for cardinal interpolation with the $m$-harmonic kernel 
\[
\Phi_0 (x) = \left\{
\begin{array}{ll}
\|x\|^{2m-d} \ln \|x\|, & \mathrm{if}\ d\ \mathrm{is\ even}, \\
\|x\|^{2m-d},  & \mathrm{if}\ d\ \mathrm{is\ odd}, 
\end{array}
\right.
\quad x\in\mathbb{R}^d,
\]
for which the cardinal symbol (\ref{eq:CI-symbol}) cannot be defined classically. 

\section{Scale independent exponential decay}
\label{sec:main}
\addtocounter{equation}{-11}

For each $h>0$, estimate (\ref{eq:kernel-decay}) implies that the kernel $\Phi_h$ decays exponentially: 
$|\Phi_h ( y )| \leq C_0\, e^{-\alpha h\|y\|}$, $y\in\rd$. Hence, by \cite[Theorem 2.7]{ab20}, this decay is 
transferred to the Lagrange function $\widetilde{\chi_h}$ as
$|\widetilde{\chi_h} ( y )| \leq A_h\, e^{-B_h \|y\|}$, $y\in \rd$, for some positive 
constants $A_h$, $B_h$ that, \emph{a priori}, may depend on $h$.

The main result of this paper, stated next, asserts that the exponential decay of $\widetilde{\chi_h}$ 
actually holds with constants independent of $h\in(0,1]$.

\begin{theorem}	\label{th:main}
There exist $A, B>0$, depending only on $d$ and $m$, such that
\ben
|\widetilde{\chi_h} ( y )| \leq A\, e^{-B |y|}, \quad y\in\rd,\ h\in(0,1],
\label{eq:scale-indep}
\een
where $|y| = |y_1| + \cdots + |y_d|$.
\end{theorem}

For $d=1$, this result was proved by Bejancu et al.\ \cite[Theorem 4.1]{bkr07} in the quite different 
setting of multivariate `polyspline' interpolation of continuous data prescribed on equally spaced 
parallel hyperplanes. In fact, in that context, $h$ does not play the role of a scaling parameter, 
denoting instead the norm of a certain frequency variable $\xi$.

As in \cite{bkr07}, our proof of Theorem~\ref{th:main} uses the Fourier transform representation of 
$\widetilde{\chi_h}$ given in Lemma~\ref{le:Lag-FT}. The main technical ingredient is 
Lemma~\ref{le:madych} below (which extends \cite[Lemma~3.1]{bkr07}, for $d=1$), based on ideas 
of Madych and Nelson \cite[Lemma 1]{mn90}. To state it, we introduce the notation:
\ben
q_h (z) := z_1^2 + \cdots + z_d^2 + h^2, \quad z=(z_1,\ldots,z_d)\in\cd,\ h\geq0.
\label{eq:squares}
\een
Also, for a given set $\Omega\subset\mathbb{R}$ and a positive number $\alpha$, we let
\[
\Omega_\alpha = \{ \zeta\in\mathbb{C} : \mathrm{Re}\,\zeta\in \Omega\ \mathrm{and}\ 
\mathrm{Im}\,\zeta\in(-\alpha,\alpha) \}.
\]
The Cartesian product of $d$ copies of $\Omega_\alpha$ is denoted by 
$\Omega_\alpha^d := (\Omega_\alpha)^d\subset\cd$.

The next lemma extends the symbol function $\omega_m (\cdot,h)$ defined by 
(\ref{eq:inv-symbol-r}) as an analytic function of $z$ on a certain tube $\rd_\alpha\subset\cd$:
\ben
\omega_m (z,h) := \left( \sum_{k\in\zd} \frac{1}{[q_h (z+2\pi k)]^m} \right)^{-1}.
\label{eq:inv-symbol}
\een
%

\begin{lemma}		\label{le:madych}
Let $Q:=[-\pi,\pi]$. There exists $\alpha:=\alpha(d,m)>0$ such that, for all $h\in[0,1]$, 
$\omega_m (z,h)$ and $\omega_m (z,h) [q_h (z)]^{-m}$ are analytic functions of 
$z$ in the common tube $\rd_\alpha$.
\end{lemma}

\begin{proof}
For the sake of exposition, we split the proof in three parts.

1. Let $\alpha\in(0,\sqrt{(\pi^2-1)/d}]$. Then we claim that, for all $h\in [0,1]$, $z\in Q_\alpha^d$, and 
$k\in\zd\setminus\{0\}$, we have $q_h (z+2\pi k)\not=0$, and the series 
\[
G_{m,h} (z) := \sum_{k\in\zd\setminus\{0\}} \frac {1} {[q_h (z+2\pi k)]^m},
\]
is absolutely and uniformly convergent, so analytic, for $z\in \mathring{Q}_\alpha^d$,
$\mathring{Q}:=(-\pi,\pi)$.

Indeed, for all $k\in\zd\setminus\{0\}$, $h\geq0$, and $z=u+iv \in Q_\alpha^d$ (i.e.\ $|u_p|\leq\pi$, 
$|v_p|\leq\alpha$, for all $p=1,\ldots,d$), we have
\ben
| q_h (z+2\pi k) | \geq \mathrm{Re} [ q_h (z+2\pi k) ] 
= \| u+2\pi k \|^2 - \|v\|^2 + h^2 \geq \pi^2 - d\alpha^2 \geq 1.
\label{eq:madych1}
\een
Also, if $\| k \| \geq \sqrt{d}$, then $\|u+2\pi k\| \geq 2\pi\|k\| - \|u\| \geq \pi (2\|k\|-\sqrt{d}) > 0$ and 
$\|v\|^2 < d\alpha^2 < \pi^2$ imply
\ben
| q_h (z+2\pi k) | \geq \| u+2\pi k \|^2 - \|v\|^2 + h^2 
> \pi^2 (2\|k\|-\sqrt{d})^2 - \pi^2.
\label{eq:madych2}
\een
Thus, using (\ref{eq:madych1}) for the terms of index $k\in\zd\setminus\{0\}$ satisfying $\| k \| < \sqrt{d}$, 
and (\ref{eq:madych2}) for the terms corresponding to $\| k \| \geq \sqrt{d}$, we obtain the estimate 
\bem
| G_{m,h} (z) | 
& \leq & \sum_{0 < \| k \| < \sqrt{d}}  | q_h (z+2\pi k) |^{-m} 
+ \sum_{\| k \| \geq \sqrt{d}}  | q_h (z+2\pi k) |^{-m}  \nonumber \\
& \leq & N + \pi^{-2m} \sum_{\| k \| \geq \sqrt{d}}  [(2\|k\|-\sqrt{d})^2 - 1]^{-m} =: M, 
\label{eq:madych3}
\eem
where $N:=N(d)$ is the number of indices $k\in\zd\setminus\{0\}$ such that $\| k \| < \sqrt{d}$, and 
$M:=M(d,m)<\infty$, since $2m>d$. This estimate implies that the above claim is true.

2. Next, we note that, by its definition (\ref{eq:squares}), $q_h (z)$ is continuous as a function of 
$(z,h)\in\cd\times[0,\infty)$, while $G_{m,h} (z)$ is continuous as a function of 
$(z,h)\in Q_\alpha^d\times[0,\infty)$, since the estimate (\ref{eq:madych3}) is valid for all $h\geq0$. 
Hence, the product $[q_h (z)]^m G_{m,h} (z)$ is continuous for $(z,h)\in Q_\alpha^d\times[0,\infty)$. 
Further, since this product is nonnegative for $z:=t\in Q^d$ and $h\geq0$, we have 
\[
1+[q_h (t)]^m G_{m,h} (t)\geq 1,\quad t\in Q^d, \ h\geq0.
\]
Using the fact that $Q^d\times [0,1]$ is compact, it follows, by reducing $\alpha$ if necessary, that
the function $1+[q_h (z)]^m G_{m,h} (z)$, which is continuous in variables $(z,h)$, remains bounded 
away from zero in modulus on the set $Q_\alpha^d\times[0,1]$, hence
\ben
1+[q_h (z)]^m G_{m,h} (z)\not=0,\quad (z,h)\in Q_\alpha^d\times [0,1].
\label{eq:madych4}
\een

3. From parts 1 and 2 above, we deduce that, for $z\in Q_\alpha^d$ and $h\in[0,1]$, the definition 
(\ref{eq:inv-symbol}) provides $\omega_m (z,h) \in \mathbb{C}$, if $q_h (z)\not=0$. We also let 
$\omega_m (z,h) := 0$, if $q_h (z)=0$. Therefore, the following identity holds:
\ben
\frac {\omega_m (z,h)} {[q_h (z)]^m} = \frac {1} {1 + [q_h (z)]^m G_{m,h} (z)}, 
\quad z\in Q_\alpha^d,\ h\in[0,1],
\label{eq:madych5}
\een
where the left-hand side is assigned the value $1$, if $q_h (z)=0$. This shows that 
$\omega_m (z,h) [q_h (z)]^{-m}$, hence $\omega_m (z,h)$ as well, are analytic functions of 
$z$ on $\mathring{Q}_\alpha^d$, for all $h\in [0,1]$. By its periodicity and continuity on $Q_\alpha^d$, 
$\omega_m (z,h)$ extends analytically to $\mathbb{R}_\alpha^d$ as a function of $z$, for all $h\in [0,1]$. 
Now, (\ref{eq:madych1}) implies that $q_h (z)\not=0$ for $z\in\mathbb{R}_\alpha^d\setminus Q_\alpha^d$, hence 
$\omega_m (z,h) [q_h (z)]^{-m}$ is also analytic on $\mathbb{R}_\alpha^d$ as a function of $z$, 
for all $h\in [0,1]$. 
\end{proof}
\medskip

\noindent
\textbf{Remark.} It is possible to establish Lemma~\ref{le:madych} for all $h\geq0$, by replacing the 
above compactness argument for (\ref{eq:madych4}) with the extension to arbitrary $d$ of the  
inequality \cite[(14)]{bkr07}.
%
\medskip

\noindent
\textbf{Proof of Theorem~\ref{th:main}.} 
Let $\alpha$ be the value provided by Lemma~\ref{le:madych}, 
which implies that the integrand of the FT representation (\ref{eq:Lag-FT}) can be extended as an 
analytic function in the common tube $\mathbb{R}_\alpha^d$, for all $h\in [0,1]$. We pick 
$B\in(0,\alpha)$ and proceed to prove (\ref{eq:scale-indep}), with some $A>0$, for all $h\in(0,1]$.

To this aim, for $y\in\rd$, we intend to estimate the integral (\ref{eq:Lag-FT}), after changing its contour of 
integration from $\rd$ to $\rd + i \gamma$, where $\gamma=(\gamma_1,\ldots,\gamma_d)$, 
$\gamma_p = \pm B$, and $\gamma_p$ has the same sign as $y_p$, for $p\in\{1,\ldots,d\}$ (this sign 
choice being irrelevant if $y_p=0$). We will also employ the notation 
\[
\Psi_{m,h} (z) := \omega_m (z,h) [q_h (z)]^{-m}, \quad z\in\mathbb{R}_\alpha^d,\ h\in [0,1].
\]

The change of contour is obtained in $d$ steps, via successive applications of Cauchy's Theorem. In the 
first step, this theorem implies
\ben
\int_{\Gamma_1} e^{i y_1 z_1} \left( \int_{\mathbb{R}^{d-1}} e^{i (y_2 t_2+\cdots+y_d t_d)} \, 
\Psi_{m,h} (z_1,t_2,\ldots,t_d) \, dt_2\ldots dt_d \right) dz_1 = 0,
\label{eq:Cauchy1}
\een
where $\Gamma_1\subset \mathbb{C}$ is the rectangle of horizontal (long) sides $[-R,R]$ and 
$[-R,R]+i\gamma_1$, and corresponding vertical sides $\pm R + i\gamma_1 [0,1]$. The integral inside 
brackets is analytic in $z_1$, since $\Psi_{m,h} (z_1,t_2,\ldots,t_d)$ has this property for each 
$(t_2,\ldots,t_d)$. Next, note that the outside integral along the two vertical sides tends to zero as 
$R\rightarrow\infty$, due to the boundedness of $\omega_m$ and the sufficient power growth in the 
denominator of $\Psi_{m,h}$. Thus, (\ref{eq:Lag-FT}) and (\ref{eq:Cauchy1}) imply
\[
\widetilde{\chi_h} (y) = \frac{-1}{(2\pi)^d }
\int_{\mathbb{R}+i\gamma_1} e^{i y_1 z_1} \int_{\mathbb{R}^{d-1}} e^{i (y_2 t_2+\cdots+y_d t_d)} \, 
\Psi_{m,h} (z_1,t_2,\ldots,t_d) \, dt_d\ldots dt_2\, dz_1 .
\]
Repeating this argument for each of the remaining $d-1$ variables, we obtain, via Fubini's Theorem,
\[
\widetilde{\chi_h} (y) = \frac{(-1)^d}{(2\pi)^d }
\int_{\rd+i\gamma} e^{i (y_1 z_1+\cdots+y_d z_d)}  \, 
\Psi_{m,h} (z_1,\ldots,z_d) \, dz.
\]

On this integration contour, we use $z_p = t_p + i\gamma_p$ for $p=1,\ldots,d$, hence 
\[
\sum_{p=1}^d y_p z_p = yt + i \sum_{p=1}^d y_p \gamma_p 
= yt + i B \sum_{p=1}^d |y_p| ,
\]
which implies the estimate:
\[
|\widetilde{\chi_h} (y)| \leq \frac {e^{-B |y|}} {(2\pi)^d}
\int_{\rd} |\Psi_{m,h} (t+i\gamma)| \, dt.
\]
Therefore, to obtain (\ref{eq:scale-indep}), it is sufficient to prove the existence of a constant $C>0$, 
independent of $h$, such that 
\ben 
\int_{\rd} \frac {|\omega_{m} (t+i\gamma,h)|} {|q_h (t+i\gamma)|^m} \, dt \leq C, \quad h\in[0,1].
\label{eq:Lag-est2}
\een
To achieve this, we estimate the above integral by splitting it over two regions: $\|t\|\leq\alpha\sqrt{d}$ 
and $\|t\|>\alpha\sqrt{d}$.

In the first region, the restriction on $\alpha$ from the first line of the proof of Lemma~\ref{le:madych} 
implies $|t_p|\leq\|t\|\leq\alpha\sqrt{d}<\pi$ for all $p=1,\ldots,d$, hence $t+i\gamma\in Q_\alpha^d$.
Since the identity (\ref{eq:madych5}) shows that $\omega_{m} (t+i\gamma,h) [q_h (t+i\gamma)]^{-m}$ 
is a continuous (hence, bounded) function of $(t,h)$ on $Q^d\times[0,1]$, it follows that there exists 
$C_1:=C_1(d,m,\alpha)>0$, such that 
\ben 
\int_{\|t\|\leq\alpha\sqrt{d}} \frac {|\omega_{m} (t+i\gamma,h)|} {|q_h (t+i\gamma)|^m} \, dt \leq C_1, 
\quad h\in[0,1].
\label{eq:Lag-est3}
\een

For $\|t\|>\alpha\sqrt{d}$, we use the fact that (\ref{eq:madych5}) also implies the continuity of 
$\omega_m (z,h)$ as a function of $(z,h)$ on $Q_\alpha^d\times[0,1]$. By the $2\pi$-periodicity of 
$\omega_m$ in each component of $z$, we deduce the existence of a constant 
$C_2:=C_2 (d,m,B)>0$, such that 
\[
|\omega_{m} (t+i\gamma,h)| \leq C_2, \quad t\in\rd,\ h\in[0,1].
\]
Using this bound, coupled with the estimate, for $\|t\|>\alpha\sqrt{d}>B\sqrt{d}$:
\[
|q_h (t+i\gamma)| \geq \mathrm{Re}\, q_h (t+i\gamma) = \|t\|^2 - dB^2 + h^2 \geq \|t\|^2 - dB^2 > 0, 
\]
we obtain, since $2m>d$,
\ben 
\int_{\|t\| > \alpha\sqrt{d}} \frac {|\omega_{m} (t+i\gamma,h)|} {|q_h (t+i\gamma)|^m} \, dt 
\leq \int_{\|t\| > \alpha\sqrt{d}} \frac {C_2\, dt} {(\|t\|^2 - dB^2)^m} \leq C_3, 
\quad h\in[0,1],
\label{eq:Lag-est4}
\een
for some constant $C_3:=C_3(d,m,\alpha,B)<\infty$.

Thus, (\ref{eq:Lag-est3}) and (\ref{eq:Lag-est4}) imply (\ref{eq:Lag-est2}), 
which completes the proof.
\ \rule{0.5em}{0.5em}

\section{Convergence rates}
\label{sec:conv}
\addtocounter{equation}{-12}

This section uses the main result to obtain a uniform bound on the Lebesgue constant for 
non-stationary cardinal interpolation with the \mat\ kernel, which eventually enables the transfer of the 
approximation order of the flat ladder $\{\mathcal{S}^h (\Phi_h)\}_h$ over to the scaled cardinal 
interpolation scheme.

For each $h\in(0,1]$, let $I_h$  denote the interpolation operator taking any bounded function $f$ 
on \rd\ to its interpolant $I_h f$ generated with the kernel $\Phi$ on the scaled grid $h\zd$, i.e.
\ben
I_h f (x) := \sum_{j\in\zd} f(hj) \chi_h (x-hj),\quad  x\in\rd.
\label{eq:scaled-scheme}
\een
Since the change of variables $y=h^{-1}x$ equivalently expresses this operator as interpolation to 
any function $f(h\cdot)$ on the standard cardinal grid \zd\ by means of the kernel $\Phi_h$, it follows 
that all basic properties of cardinal interpolation on \zd\ (e.g. \cite[\S2.1]{ab20}) also apply to $I_h$. 
In particular, the norm (or `Lebesgue constant') of $I_h$ as a linear bounded operator on $L_\infty (\rd)$ 
can be expressed as:
\ben
\| I_h \|_\infty = \sup_{x\in\rd} \sum_{j\in\zd} | \chi_h (x-hj) |.
\label{eq:scaled-Leb}
\een
Thus, the scale-independent decay of $\widetilde{\chi_h} = \chi_h (h\cdot)$ established in the previous 
section leads to the following immediate consequence for Lebesgue constants.

\begin{corollary}	\label{cor:Lebesgue}
The Lebesgue constant $\| I_h \|_\infty$ for non-stationary interpolation with the \mat\ kernel $\Phi$ 
on the scaled grid $h\zd$ admits a uniform bound for all $h\in(0,1]$, i.e.\ there exists 
$C:=C(d,m,\alpha,B)>0$, such that 
\ben
\| I_h \|_\infty \leq C,\quad  h\in(0,1].
\label{eq:Leb-bound}
\een
\end{corollary}

\begin{proof}
Since $\chi_h (x) = \widetilde{\chi_h} (h^{-1} x)$, from (\ref{eq:scaled-Leb}) and 
Theorem~\ref{th:main} we obtain:
\[
\| I_h \|_\infty \leq \sup_{x\in\rd} \sum_{j\in\zd} A e^{-B\|h^{-1}x-j\|} 
= A\, \sup_{y\in\rd} \sum_{j\in\zd} e^{-B\|y-j\|} =: C < \infty,
\]
as required.
\end{proof}
\bigskip

The next theorem specializes to the \mat\ kernel $\Phi$ a result of de Boor and Ron \cite{br92} on the 
approximation order of the non-stationary ladder $\{\mathcal{S}^h (\Phi_h)\}_h$. We employ the notation
\[
\mathrm{dist}(f,\mathcal{A};X) := \inf_{s\in\mathcal{A}} \| f-s \|_X,
\]
for the distance from a function $f$ to a set of functions $\mathcal{A}$, measured in the norm of a 
space $X$. Also, we denote by $L^{2m,1}$ the Bessel potential space consisting of all functions 
$f\in L_1(\rd)$, such that $(1+\|\cdot\|^{2})^m\widehat{f} \in L_1 (\rd)$. Note that the Schwartz class 
of rapidly decaying smooth functions on \rd\ is a subspace of $L^{2m,1}$. Further, 
we let 
\ben
\mathcal{S}^h_0 (\Phi_h) := \mathrm{span} \{ \Phi_h (h^{-1}\cdot - j) : j\in\zd \},
\label{eq:span}
\een
the space of finite linear combinations of the translates $\{ \Phi (\cdot - hj) : j\in\zd \}$.

\begin{theorem}	\label{th:AOFL}
If $f\in L^{2m,1}$, there exists $C_f:=C(f,d,m)>0$, such that
\ben
\mathrm{dist} (f,\mathcal{S}^h_0 (\Phi_h); L_\infty) 
= \mathrm{dist} (f,\mathcal{S}^h (\Phi_h); L_\infty) 
\leq C_f \, h^{2m},\quad h\in(0,1].
\label{eq:AOFL}
\een
\end{theorem}

\begin{proof}
The right-side inequality expresses the fact that $\{\mathcal{S}^h (\Phi_h)\}_h$ provides 
$L_\infty$-approximation of order $2m$. This result follows from \cite[Theorem~2.37]{br92}, which 
holds under three general assumptions. Firstly, the family $\{\Phi_h\}_h$ should satisfy the so-called 
\emph{synthesis condition of order $2m$}. 
In the case of the \mat\ kernel, this condition is equivalent to the existence of $\delta>0$, such that 
\ben
\sup_{\| t \| \leq \delta} \sum_{j\in\zd\setminus0} 
\frac {|\widehat{\Phi_h} (ht + 2\pi j)|} {|\widehat{\Phi_h} (ht)|} = O (h^{2m}).
\label{eq:synthesis}
\een
Since $2m>d$, (\ref{eq:synthesis}) is satisfied, for a sufficiently small $\delta$, due to the estimate:
\[
\left| \frac {\widehat{\Phi_h} (ht + 2\pi j)} {\widehat{\Phi_h} (ht)} \right| 
= \frac {h^{2m} (1+\|t\|^2)^m} {(h^2 + \|ht+2\pi j\|^2)^m} 
\leq \frac {h^{2m} (1+\delta^2)^m} {\|ht+2\pi j\|^{2m}} 
\leq \frac {h^{2m} (1+\delta^2)^m} {(2\pi\|j\|-\delta)^{2m}} .
\]
A second assumption that needs to be verified is the boundedness of the semi-discrete convolution 
operator generated by $\Phi_h$, for each $h$. Using \cite[Proposition~2.3]{br92}, this is ensured by 
condition (\ref{eq:minimal}), due to the exponential decay of $\Phi_h$. Thirdly, the 
function $f$ needs to be $2m$-`admissible', i.e.\ $(1+\|\cdot\|^{2m})\widehat{f} \in L_1 (\rd)$, which is 
easily verified for any $f\in L^{2m,1}$.

As for the left-side equality of (\ref{eq:AOFL}), which transfers the rate of convergence to approximation 
with finite, rather than infinite, linear combinations of the translates $\{ \Phi (\cdot - hj) : j\in\zd \}$, this 
follows from a result of Johnson \cite[Proposition~2.2]{mjj97a}, due to condition (\ref{eq:minimal}) 
satisfied by $\Phi_h$, and to the fact that any function $f\in L^{2m,1}$ is necessarily bounded and has 
the limit $0$ at infinity.
\end{proof}
\bigskip

\noindent
\textbf{Remarks.} 
(i) The convergence rate $2m$ for approximation from a non-stationary ladder 
generated by the \mat\ kernel $\Phi=\Phi_{m,d}$ defined in (\ref{eq:matern}) also follows from the more 
general result of Johnson \cite[Theorem~3.7]{mjj00}. That result replaces the \zd-shifts by a set of 
translations $\Xi$ which is a sufficiently small perturbation of \zd, it considers errors in $L_p$-norms, 
for all $1\leq p\leq \infty$, and it applies to sufficiently smooth functions $f$ from a Besov space.

(ii) Note that, for approximation of functions in Bessel potential spaces by \emph{finite} linear 
combinations of quasi-uniform translates of $\Phi_{m,d}$, Ward \cite[\S3.1]{ward13} and Ward and 
Unser \cite[\S3.2]{WU14} obtain the convergence rate $O(h^{2m})$, if $d$ is odd, and the slower rate 
$O(h^{2m-1})$, if $d$ is even. Hence, Theorem~\ref{th:AOFL} improves this rate for even $d$ and 
a uniform grid. In \cite{ward13} and \cite{WU14}, $L_p$-error bounds, as well as kernels of non-integer 
order, are also considered. 
%
\medskip

The uniform bound on the Lebesgue constant given in Corollary~\ref{cor:Lebesgue} can now be used to 
transfer the convergence order of Theorem~\ref{th:AOFL} to non-stationary interpolation on the grid $h\zd$ 
with the \mat\ kernel $\Phi$.

\begin{corollary}	\label{cor:AOCI}
If $f\in L^{2m,1}$, there exists a constant $\widetilde{C}_f:=\widetilde{C} (f,d,m,\alpha,B)$, such that the 
interpolant \emph{(\ref{eq:scaled-scheme})} satisfies
\ben
\| f - I_h f \|_\infty \leq \widetilde{C}_f \, h^{2m},\quad h\in(0,1].
\label{eq:AOCI}
\een
\end{corollary}

\begin{proof}
We employ a classical estimate based on the Lebesgue constant, which relates the error at 
interpolating $f$ by $I_h f$ on $h\zd$ to the error for \emph{approximating} $f$ by any 
$s_h\in\mathcal{S}^h (\Phi_h)$:
\begin{eqnarray*}
\| f - I_h f \|_\infty & \leq & \| f - s_h \|_\infty + \| s_h - I_h f \|_\infty   \\ 
& \leq & ( 1 + \| I_h \|_\infty )\, \| f - s_h \|_\infty,
\end{eqnarray*}
where we have used the projection property $I_h s_h = s_h$ (see \cite[Corollary 2.2]{ab20}).
Now (\ref{eq:Leb-bound}) and (\ref{eq:AOFL}) imply (\ref{eq:AOCI}). 
\end{proof}

\section{Compactly supported perturbation kernels}
\label{sec:comp-supp}
\addtocounter{equation}{-7}

In this section, we transfer the approximation result of Theorem~\ref{th:AOFL} to the compactly 
supported radial kernels introduced and studied independently by Johnson \cite{mjj12} and 
Ward and Unser \cite{WU14}, and we discuss examples of such kernels.

We start by recalling (e.g.\ \cite[Theorem~5.26]{wend05}) that, if $\Psi=\psi (\|\!\cdot\!\|)$ is an integrable 
radial function on \rd, with profile $\psi$ defined on $[0,\infty)$, then its Fourier transform 
$\widehat{\Psi}$ is also radial, namely $\widehat{\Psi} = (F_d\psi) (\|\!\cdot\!\|)$, with profile given by
\[
(F_d\psi) (r) = r^{1-\frac{d}{2}} \int_0^\infty \psi (t) \, t^{d/2} J_{\frac{d}{2} - 1} (rt) \, dt, 
\quad r>0, 
\]
where $J_\nu$ is the Bessel function of the first kind of order $\nu$.

For $m>d/2$, Ward and Unser \cite{WU14} defined the class of continuous profiles 
$\psi : [0,\infty)\to\mathbb{R}$, such that the $d$-dimensional Fourier transform of the radial 
function $\Psi:=\psi(\|\!\cdot\!\|)$ has a profile of the form 
\ben
(F_d\psi) (r) = C r^{-2m}\lambda (r), \quad r>0,
\label{eq:perturbFT}
\een
where $C>0$ and $\lambda : [0,\infty)\to\mathbb{R}$ satisfies three conditions: 
\begin{enumerate}
\item \label{C1}
There exist a positive integer $K$, a set of nodes $0<r_1<r_2<\cdots<r_K$, and a sequence of 
real coefficients $\{a_j\}_{j=1}^K$, such that: 
\ben
\lambda (r) = 1 + \sum_{j=1}^K a_j (r_j r)^{1-\frac{d}{2}} J_{\frac{d}{2}-1} (r_j r),\quad r\geq 0.
\label{eq:lambda}
\een
\item \label{C2}
There exists $\lim_{r\to 0^+} (F_d\psi) (r) =:\beta > 0$.
\item \label{C3}
$\lambda(r) > 0$ for $r>0$.
\end{enumerate}
As detailed in the last part of this section, examples of profiles $\psi$ satisfying the above conditions 
have first been constructed by Johnson \cite{mjj12}. 

Building on the Paley-Wiener approach of Baxter \cite{bax08,bax11} for kernel engineering, Ward and Unser 
\cite[Proposition 2.2]{WU14} proved that conditions \ref{C1} and \ref{C2} above imply that $\psi$ must 
be compactly supported. Further, in \cite[Lemma~2.3]{WU14}, they showed that a radial function $\Psi$ 
whose Fourier transform profile has the form (\ref{eq:perturbFT}) is a `perturbation' of $\Phi$, the 
$d$-dimensional \mat\ kernel (\ref{eq:matern}) of corresponding parameter $m$, in the sense that it 
can be expressed as a convolution 
\ben
\Psi = \mu * \Phi,
\label{eq:perturb}
\een
for an invertible finite Borel measure $\mu$. Using this fact, \cite[Theorem 3.11]{WU14} proved 
that such a kernel $\Psi$ satisfies similar approximation properties as $\Phi$.
Consequently, the convergence rate of Theorem~\ref{th:AOFL} can also be transferred to approximation 
by finite linear combinations of translates of $\Psi_h$.

\begin{corollary}	\label{cor:AOPSI}
Let $\Psi$ be a compactly supported kernel, with a Fourier transform profile of the form 
\emph{(\ref{eq:perturbFT})}. If $f\in L^{2m,1}$, there exists $\check{C}_f:=\check{C}(f,d,m)>0$, 
such that
\ben
\mathrm{dist} (f,\mathcal{S}^h_0 (\Psi_h); L_\infty) 
 \leq \check{C}_f \, h^{2m},\quad h\in(0,1],
\label{eq:AOPSI}
\een
where $\mathcal{S}^h_0 (\Psi_h)$ is defined analogously to \emph{(\ref{eq:span})}.
\end{corollary}

\begin{proof}
Let $\nu:=\mu^{-1}$, where $\mu$ is the invertible finite Borel measure from (\ref{eq:perturb}).
The proof of (\ref{eq:AOPSI}) follows from the estimates given in the proof of \cite[Theorem 3.11]{WU14}, 
which relate the error of approximating $f$ from $\mathcal{S}^h_0 (\Psi_h)$ to the error of approximating 
$\nu*f$ from $\mathcal{S}^h_0 (\Phi_h)$. Thus, one only needs to ensure that $f\in L^{2m,1}$ implies 
$\nu*f\in L^{2m,1}$, which is seen to hold, due to the boundedness of the Fourier transform $\widehat{\nu}$.
\end{proof}
\bigskip

\noindent
\textbf{Remark.} 
For $d\geq 3$, the approximation order $2m$ from the non-stationary ladder generated by the 
\emph{perturbed} shifts of the kernel $\Psi$ can also be obtained directly from the general result of Johnson 
\cite[Theorem~3.7]{mjj00} described after Theorem~\ref{th:AOFL}. Indeed, for this kernel, the hypotheses 
of Johnson's theorem are seen to be verified due to the form of the Fourier profile (\ref{eq:perturbFT}) 
and the asymptotic properties of the Bessel functions which appear in (\ref{eq:lambda}).
\bigskip

In the remaining part of this section, we show that two of the families of profiles constructed by Johnson 
\cite{mjj12}, for $d=2$ and $d=3$, belong to the above class defined by Ward and Unser. Note that, in 
place of the Paley-Wiener approach, these constructions use an L-spline approach based on the finite 
dimensional representations of piecewise polyharmonic radial functions developed in \cite{mjj98}.
\bigskip

\noindent
\textbf{Johnson's compactly supported profiles $\{\eta_m\}$ for $d=2$.} 
This family of profiles is described in \cite[\S3]{mjj12}. For each integer $m\geq 1$, $\eta_m$ is a 
piecewise defined function with nodes at $0, 1, 2,\ldots, m$, such that $\eta_m = 0$ on $(m,\infty)$, 
$\eta_m \in C^{2m-2} (0,\infty)$, and each non-trivial piece of $\eta_m$ belongs to the $2m$-dimensional 
null-space of $L^m$, where $L= \frac{d^2}{dr^2} + \frac{1}{r} \frac{d}{dr}$ is the radial Laplacean in 
$\mathbb{R}^2$.
Up to a suitable normalization, $\eta_m$ is uniquely determined by imposing certain $m-1$ boundary 
conditions on its restriction to the first interval $(0,1)$. Therefore $\eta_m (\|\!\cdot\!\|)$ is a 
compactly supported, piecewise $m$-harmonic, radially symmetric function on $\mathbb{R}^2$. It follows  
from \cite[(3.4)]{mjj12} that (\ref{eq:lambda}) holds with $K=m$ for $\lambda:=(\cdot)^{2m} F_2 \eta_m$, 
since
\ben
(F_2 \eta_m) (r) = \frac{4^{m-1}[(m-1)!]^2}{r^{2m}} 
\left( 1+ \sum_{j=1}^m a_j J_0 (jr) \right), \quad r>0,
\label{eq:etaFT}
\een
where $\{a_j\}_{j=1}^m$ are uniquely determined by the fact that
 $\lim_{r\to0^+} (F_2 \eta_m) (r)=:\beta\in\mathbb{R}$ 
exists (as $\eta_m$ has compact support). Next, using an integral representation of $J_0$, Johnson 
\cite[(3.5)]{mjj12} proves the remarkable property that $(F_2 \eta_m) (r)>0$ for $r>0$, i.e.\ condition 
\ref{C3} holds for $\lambda$. Since $K=m$ and $F_2 \eta_m$ is an entire function, it also follows that 
$\beta>0$ is satisfied automatically, hence $\lambda$ satisfies all three conditions listed after 
(\ref{eq:perturbFT}). 

For $m=2$, the two nontrivial pieces of $\eta_2$ are displayed explicitly in \cite{mjj12} as:
\ben
\eta_2 (t) = \frac{1}{3} \left\{ 
\begin{array}{ll}
4\ln 2 + (\ln 2 - 3) t^2 + 3t^2 \ln t, & t\in(0,1], \\
(4\ln 2-4) - 4\ln t + (\ln 2 + 1) t^2 - t^2 \ln t, & t\in(1,2].
\end{array}
\right.
\label{eq:eta2}
\een
The scaled profile $\frac{3}{4\ln2} \eta_2 (2t)$, supported on $[0,1]$, was also obtained by Ward and 
Unser in \cite[Example 2.4]{WU14}. Further, $\eta_2(\|\!\cdot\!\|)$ was identified in \cite{ab16} as a 
radially symmetric thin plate spline.

\bigskip

\noindent
\textbf{Johnson's compactly supported profiles $\{\psi_{3,m}\}$ for $d=3$.} 
For each integer $m\geq1$, let $\psi_m$ be the restriction to $[0,\infty)$ of the polynomial B-spline 
which has a double knot at $0$ and simple knots at $\pm1,\ldots,\pm m$. This family of B-splines 
was studied by Al-Rashdan and Johnson \cite{rj12}, who proved:
\ben
(F_1 \psi_m) (r) = \frac{\delta_m}{r^{2m}} 
\left( 1+ \sum_{j=1}^m b_j \frac{\sin (jr)}{r} \right) > 0, \quad r>0,
\label{eq:psiFT}
\een
where $\delta_m>0$ and the coefficients $\{b_j\}_{j=1}^m$ are uniquely determined by the fact that 
$\lim_{r\to0^+} (F_1 \psi_m) (r)=:\beta\in\mathbb{R}$ exists. 

As part of a larger class of profiles, the family $\{\psi_{3,m}\}$ was defined by Johnson \cite[\S8]{mjj12} 
via
\ben
\psi_{3,m} := \mathcal{D} \psi_m, \quad m\geq1,
\label{eq:psi3m}
\een
where $\mathcal{D} = - \frac{1}{r} \frac{d}{dr}$. Hence, the `dimension walk' formula 
$F_3 (\mathcal{D} \psi_m) = F_1 \psi_m$, the identity
$r^{-1/2} J_{1/2} (r) = \sqrt{\frac{2}{\pi}} \,\frac{\sin r}{r}$, and the last two displays imply 
\ben
(F_3\psi_{3,m}) (r) = \frac{\delta_m}{r^{2m}} 
\left( 1+ \sqrt{\frac{\pi}{2}} \sum_{j=1}^m jb_j (jr)^{-1/2} J_{1/2} (jr) \right) > 0, \quad r>0.
\label{eq:psi3mFT}
\een
It follows that, for $d=3$, the profile $F_3\psi_{3,m}$ is of the form (\ref{eq:perturbFT}), with $K=m$. 
Further, by \cite[Theorem 6.1]{mjj12}, $\psi_{3,m} (\|\!\cdot\!\|)$ is a 
compactly supported, piecewise $m$-harmonic, radially symmetric function on $\mathbb{R}^3$. 

In the special case $m=2$, \cite{rj12} gives the explicit expression
\ben
\psi_2 (t) = \left\{ \begin{array}{ll}
8 - 24 t^2 + 24 t^3 - 7 t^4, & t\in[0,1], \\
(2 - t)^4 , & t\in(1,2],
\end{array}
\right.
\label{eq:psi2}
\een
from which the expression of $\psi_{3,2}$ can be calculated via (\ref{eq:psi3m}). The scaled version 
$\frac{1}{48} \psi_{3,2} (2r)$ is also provided as Example 2.5 by Ward and Unser \cite{WU14}, while 
their Example 2.6 is seen to coincide with $\frac{1}{540} \psi_{3,3} (3r)$.
\bigskip

\noindent
\textbf{Remarks.}  
(i) It follows from the proofs of \cite[Theorem 2.8]{rj12} and \cite[Theorem 3.11]{mjj12}, that the three 
conditions imposed on the function $\lambda$ of (\ref{eq:perturbFT}) imply that the corresponding 
profile $\psi$ has `Sobolev regularity' $(d,m)$, i.e. 
\[
A (1+r^2)^{-m} \leq (F_d \psi) (r) \leq B (1+r^2)^{-m}, \quad r>0,
\]
for some constants $B\geq A>0$. Hence, a natural application of such profiles $\psi$ may occur in the 
field of multilevel interpolation algorithms; cf.\ Farrell et al.\ \cite{fgw17}.  Another potential application, 
to tomographic image reconstruction via X-ray transform, is discussed by Ward and Unser  
\cite{WU14}, where the plots of the profiles $\eta_2$ and $\psi_{3,2}$, rescaled on the support 
interval $[0,1]$, are also presented. 

(ii) Since a function $\lambda$ of the form (\ref{eq:lambda}) is even and entire, the condition that 
$\lim_{r\to0^+} (F_d \psi) (r)$ exists is seen to imply $K\geq m$. For $K>m$, currently there do not 
seem to be any constructions of such functions $\lambda$ satisfying the three conditions described 
at the beginning of the section. For $K=m$, as shown by Johnson \cite[(3.2)]{mjj12}, the coefficients 
$\{a_j\}$ of (\ref{eq:lambda}) are uniquely determined by the requirement that 
$\lim_{r\to0^+} (F_d \psi) (r)$ exists. Hence, it is remarkable that the above two families of profiles 
$\{\eta_m\}$ and $\{\psi_{3,m}\}$, for $d=2$ and $d=3$, also satisfy the extra condition that 
$\lambda(r)>0$ for $r>0$. As already noted above, this automatically ensures 
$\lim_{r\to0^+} (F_d \psi) (r)>0$.
\bigskip

\noindent
\textbf{Acknowledgements.} 
I am grateful to Michael Johnson for persuading me to consider the non-stationary flat ladder 
and suggesting that Theorem~\ref{th:AOFL} could be obtained from \cite{br92}. I also thank 
Thomas Hangelbroek for several stimulating discussions on \mat\ kernel interpolation.

\end{document}